\documentclass[11pt,a4paper]{amsart}
\usepackage{graphicx,multirow,array,amsmath,amssymb}
\usepackage{xcolor}
\usepackage{longtable}
\usepackage{supertabular} 

\newtheorem{theorem}{Theorem}

\newtheorem{lemma}[theorem]{Lemma}
\newtheorem{corollary}[theorem]{Corollary}
\theoremstyle{definition}
\newtheorem{definition}{Definition}
\theoremstyle{remark}

\begin{document}
	
\title{Four-page index and linear upper bounds for ribbonlength}

\author[H. Yoo]{Hyungkee Yoo}
\address{Department of Mathematics Education, Sunchon National University, Suncheon 57922, Republic of Korea}
\email{hyungkee@scnu.ac.kr}



\begin{abstract}
We introduce the four-page index of a knot or link as a presentation invariant arising from embeddings in a four-page open book decomposition. 
Using spanning trees of the checkerboard graph of a reduced non-split diagram, we construct a Kauffman state consisting of a single state circle. 
The associated Eulerian tour of the underlying 4-valent plane graph determines a binding circle intersecting each edge exactly once, producing a four-page presentation with at most $2c(K)$ arcs. 
Hence
$$
\alpha_4(K) \le 2c(K),
$$
with strict inequality in the non-alternating case.

We further prove that ribbonlength is bounded above by the four-page index,
and therefore obtain the linear bound
$$
\mathrm{Rib}(K) \le 2c(K).
$$
This improves the previously known general linear upper bound for ribbonlength
and provides a diagrammatic method for estimating ribbonlength.
\end{abstract}

\maketitle

\section{Introduction} \label{sec:intro}

Ribbonlength, introduced by Kauffman~\cite{K}, measures the geometric complexity of a knot or link via flat folded ribbon realizations. 
Understanding its growth in terms of crossing number has been a central question in geometric knot theory.

A folded ribbon knot is obtained by folding a long and thin rectangular strip into a planar configuration whose core is a piecewise linear diagram representing the given knot type, as illustrated in Figure~\ref{fig:folded}. 
This construction provides a planar ribbon realization of knots and links.

\begin{figure}[h!]
	\includegraphics{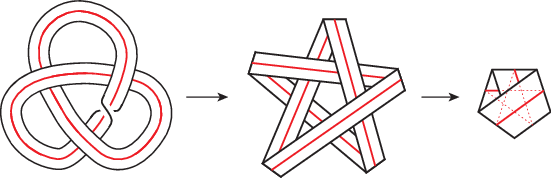}
	\caption{A folded ribbon knot of the trefoil knot}
	\label{fig:folded}
\end{figure}

Since the trivial knot admits folded ribbon realizations of arbitrarily small length, ribbonlength is defined using an infimum. 
The following formal definitions were given in~\cite{DHLM,DKTZ}.

\begin{definition}
Let $K$ be a knot. 
A {\it folded ribbon knot} $K_w$ of width $w>0$ satisfies:
\begin{enumerate}
\item The ribbon is flat and its fold lines are disjoint.
\item The core of $K_w$ is a finite union of straight line segments with crossing information representing the knot type of $K$.
\end{enumerate}
\end{definition}

\begin{definition}
Let $K$ be a knot or link and let $w>0$.
Then
$$
\mathrm{Rib}(K)
=
\inf_{K'_w \in [K]_w}
\frac{\mathrm{Len}(K'_w)}{w}
$$
is called the {\it ribbonlength} of $K$, where $[K]_w$ denotes the set of folded ribbon knots of width $w$ representing $K$.
\end{definition}

A knot or link diagram is a regular projection of a knot or link onto the plane equipped with over-under crossing information, and the crossing number $c(K)$ is the minimal number of crossings among all diagrams representing $K$.
Tian~\cite{T} established a quadratic upper bound for ribbonlength, and Denne~\cite{D} improved this to an order $c(K)^{3/2}$ estimate.
Subsequent work confirmed linear growth for various families of knots and links~\cite{DHLM,KNY1,KNY2}.

Diao and Kusner conjectured that ribbonlength admits asymptotic bounds of the form
$$
c_1 c(K)^{\alpha} \le \mathrm{Rib}(K) \le c_2 c(K)^{\beta}.
$$
They expected that $\alpha=1/2$ and $\beta=1$.
More recently, Denne and Patterson~\cite{DP} showed that the lower bound in this problem is constant.
In particular, Kim, No, and Yoo~\cite{KNY3} proved that for any knot or link $K$,
$$
\mathrm{Rib}(K) \le \frac{5}{2}c(K)+1,
$$
thereby establishing a general linear upper bound.

We approach this problem by connecting ribbonlength
to presentation invariants arising from
open book decompositions of $\mathbb{R}^3$.
An \emph{arc presentation} of a knot or link $K$ is an embedding of $K$
into finitely many pages of an open book decomposition of $\mathbb{R}^3$
so that each page contains exactly one properly embedded arc.
The minimal number of pages required is called the
\emph{arc index} $\alpha(K)$~\cite{Cr}.
A \emph{three-page presentation} of $K$ is an embedding of $K$
into three pages of an open book decomposition,
where each page may contain several pairwise disjoint arcs.
The minimal number of arcs in such a presentation
is called the \emph{three-page index} $\alpha_3(K)$~\cite{JLY}.

Building on these notions, we define the \emph{four-page index}
$\alpha_4(K)$ to be the minimal number of arcs
in a four-page presentation of $K$,
as illustrated in Figure~\ref{fig:four}.
The precise definition and basic properties of $\alpha_4(K)$
are given in Section~\ref{sec:four}.

\begin{figure}[h!]
	\includegraphics{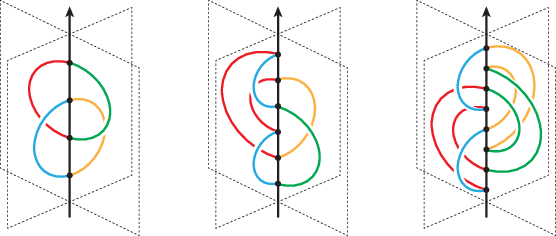}
	\caption{Four-page presentations of Hopf link, trefoil knot, and figure eight knot}
	\label{fig:four}
\end{figure}

The purpose of this paper is to present a new diagrammatic approach
that yields a sharper linear estimate.
Our method is based on a structural connection between knot diagrams,
Kauffman state, and presentation invariants.
Starting from a non-split link diagram,
we associate a suitable Kauffman state
that leads naturally to a four-page presentation.

Our first result establishes a structural relationship between the four-page index and ribbonlength.

\begin{theorem}\label{thm:rib_alpha4}
For any knot or link $K$,
$$
\mathrm{Rib}(K) \le \alpha_4(K).
$$
\end{theorem}

Since $\alpha_4(K) \le \alpha_3(K)$ and it is known that
$\alpha_3(K)$ grows linearly with the crossing number~\cite{Y3},
it follows formally that $\alpha_4(K)$ also admits
a linear upper bound.
However, such an indirect argument does not provide
a sharp coefficient.
Our main diagrammatic result is the following.

\begin{theorem}\label{thm:alpha4_bound}
For any non-split, nontrivial knot or link $K$,
$$
\alpha_4(K) \le 2c(K).
$$
Moreover, if $K$ is non-alternating, then the inequality is strict.
\end{theorem}

In particular, the bound $\alpha_4(K) \le 2c(K)$ is attained for the Hopf link and the trefoil knot.
Together with Theorem~\ref{thm:rib_alpha4},
this immediately implies the following corollary.

\begin{corollary}
For any non-split, nontrivial knot or link $K$,
$$
\mathrm{Rib}(K) \le 2c(K)
$$
with strict inequality in the non-alternating case.
\end{corollary}

The paper is organized as follows.
In Section~\ref{sec:four}, we introduce four-page presentations and define the four-page index $\alpha_4(K)$, establishing its basic properties.
In Section~\ref{sec:ribbon}, we prove that ribbonlength is bounded above by the four-page index.
Section~\ref{sec:diagram} develops a diagrammatic construction of four-page presentations from knot diagrams and establishes the linear bound $\alpha_4(K) \le 2c(K)$.
Finally, in Section~\ref{sec:circular}, we study circular four-page presentations and refine the inequality in the non-alternating case.

\section{Four-page presentations and the four-page index}
\label{sec:four}

Structured embeddings of links in $\mathbb{R}^3$ provide a combinatorial framework for studying topological complexity.
A classical example is the arc presentation introduced by Cromwell~\cite{Cr},
in which a knot or link is embedded into finitely many pages of an open book decomposition of $\mathbb{R}^3$.
The minimal number of pages required for such a representation is called the \emph{arc index} $\alpha(K)$,
which has been extensively studied in relation to crossing number, bridge number, and other invariants~\cite{BP,Cr2,CrNut,Mat,MB,Ng}.

A different viewpoint was proposed by Dynnikov~\cite{Dyn}, who introduced
\emph{three-page presentations}.
In this setting, the ambient space is restricted to three pages of an open book,
and each page may contain several mutually disjoint arcs.
Dynnikov showed that every link admits a three-page presentation, and later studies have examined this type of presentation
for various classes of knots, links and spatial graphs~\cite{Dyn2, Dyn3, JLY, Kur, Kur2, KV, Y}.

Recently, variants of three-page presentations have also been used
to study ribbonlength.
For an oriented link, one may consider
\emph{rotated three-page presentations},
where the orientation induces a cyclic order
in which the link traverses the three pages.
This construction was introduced to derive ribbonlength estimates
by passing from knot diagrams to presentation data
and subsequently to folded ribbon realizations~\cite{Y2}.
These results indicate that presentation invariants
may capture geometric information relevant to ribbonlength.
Motivated by this line of research, we introduce an analogous notion based on four pages.

\begin{definition}
A \emph{four-page presentation} of a knot or link $K$ is a presentation of $K$
in an open book decomposition of $\mathbb{R}^3$ satisfying the following conditions:
\begin{enumerate}
\item Exactly four pages meet $K$.
\item Each page contains finitely many pairwise disjoint properly embedded arcs.
\end{enumerate}
\end{definition}

Dynnikov proved that every knot or link admits a three-page presentation.
Starting from such a presentation, one may obtain a four-page presentation
as follows.
Choose an arc lying in one of the three pages and push an interior point
of that arc to the binding axis through a planar isotopy supported entirely
within the same page.
This operation splits the arc into two arcs meeting the binding axis at the new point,
thereby increasing the total number of arcs by one.
One of these arcs may then be moved into an initially empty fourth page.
Hence every knot or link admits a four-page presentation.

As in the three-page setting, the number of intersection points with the binding axis
is equal to the number of arcs in the presentation.

\begin{definition}
The \emph{four-page index} $\alpha_4(K)$ of a knot or link $K$
is defined as the minimal number of arcs
among all four-page presentations of $K$.
\end{definition}

The four-page index is closely related to the arc index and the three-page index.

First, let $K$ be a knot or link admitting a four-page presentation with $m$ arcs.
By a small perturbation, one may redistribute the arcs so that each arc lies
in a distinct page of a sufficiently large open-book decomposition.
Thus, a four-page presentation with $m$ arcs induces an arc presentation
with $m$ arcs.
Taking the minimum over all four-page presentations yields
$$
\alpha(K) \le \alpha_4(K),
$$
where $\alpha(K)$ denotes the arc index.

On the other hand, let $K$ be a nontrivial knot or link admitting a three-page presentation.
In any such presentation, the number of arcs on each page
is bounded below by the bridge number of $K$, as shown in~\cite{JLY}.
Since the bridge index of a non-trivial link is at least two,
every three-page presentation contains at least two arcs on each page.
Therefore, one may move a single arc from one of the pages
into an additional fourth page without increasing the total number of arcs.
This produces a four-page presentation with the same number of arcs, and hence
$$
\alpha_4(K) \le \alpha_3(K).
$$

\section{From four-page presentations to ribbonlength}
\label{sec:ribbon}

In this section we prove Theorem~\ref{thm:rib_alpha4}.
The key observation is that a four-page presentation
admits a direct geometric realization as a folded ribbon knot.
We now describe this construction.

\begin{proof}[Proof of Theorem~\ref{thm:rib_alpha4}]
Let $K$ be a knot or link admitting a four-page presentation with
$m$ arcs.
The number of arcs coincides with the number of intersection points
with the binding axis.
We construct a unit-width folded ribbon realization of $K$
whose core length can be made arbitrarily close to $m$.

Consider a four-page presentation of $K$.
At each binding point, exactly two arcs of the presentation meet.
Since the four pages are arranged cyclically around the binding axis,
the two adjacent arcs at a binding point are supported either in adjacent pages
or in opposite pages.
In the former case the two arcs meet at a right angle,
while in the latter case they are collinear and form a straight angle.

To each binding point we associate a square ribbon piece of side length $1$.
If the two adjacent arcs lie in opposite pages,
the square is left flat so that the ribbon continues straight,
realizing a $180^\circ$ turn.
If the two arcs lie in adjacent pages,
we fold the square along one of its diagonals,
thereby realizing a $90^\circ$ turn of the ribbon,
as illustrated in Figure~\ref{fig:folding}.

\begin{figure}[h!]
	\includegraphics{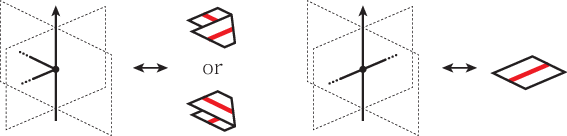}
	\caption{Local configurations near a binding point and the corresponding ribbon pieces}
	\label{fig:folding}
\end{figure}

The square pieces are placed along the binding axis
with pairwise disjoint interiors.
Since the ribbon is considered as a two–dimensional object,
these pieces may be compressed along the direction of the binding axis,
so that their heights become arbitrarily small.
In this way the entire configuration can be realized in a single plane.

The pieces are then connected within each page
according to the combinatorial structure of the four-page presentation.
Because arcs in each page are pairwise disjoint,
these connections can be performed without introducing additional intersections,
and the fold lines remain disjoint.

The resulting object is a folded ribbon knot representing $K$.
Each binding point contributes one square of side length $1$,
and the extra length introduced by connections
can be made arbitrarily small.
Hence, for every $\varepsilon>0$,
$$
\mathrm{Rib}(K) \le m+\varepsilon.
$$
Since ribbonlength is defined as an infimum over all folded ribbon realizations,
we conclude that
$$
\mathrm{Rib}(K) \le m.
$$
Taking the minimum over all four-page presentations yields
$$
\mathrm{Rib}(K) \le \alpha_4(K),
$$
which completes the proof.
\end{proof}

\section{Circular four-page presentations}
\label{sec:circular}

In this section we introduce a planar description of four-page presentations via binding circles.
In a previous work~\cite{Y3}, a method was introduced to construct three-page presentations directly from knot diagrams by means of suitable binding circles arising from spanning structures of the underlying plane graph.
We adapt and refine this construction to obtain binding circles suitable for circular four-page presentations.
This planar viewpoint provides a convenient diagrammatic method for constructing four-page presentations from link diagrams.

We consider the one-point compactification of $\mathbb{R}^3$
equipped with an open-book structure with four pages.
Under this compactification, the binding axis becomes an unknotted circle,
called the \emph{binding circle}, and each page becomes a disk
bounded by this circle as drawn in Figure~\ref{fig:circular}.
Thus a four-page presentation is completely determined
by the intersection of the link with the binding circle
together with the assignment of the resulting arcs to the four pages.

\begin{figure}[h!]
	\includegraphics{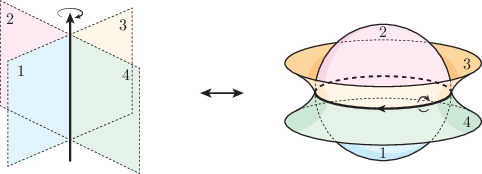}
	\caption{A four-page presentation and its circular four-page presentation with an ordering of the pages}
	\label{fig:circular}
\end{figure}

This viewpoint admits a purely planar description.
Fix a simple closed curve $\gamma$ in the plane and regard it
as the binding circle.
We place two pages inside $\gamma$ and two pages outside $\gamma$.
In contrast to the three-page case, crossings can occur both inside
and outside $\gamma$, since two distinct pages are assigned to each
side of the binding circle.
With a suitable local over-under convention distinguishing the two
pages on each side of $\gamma$, the arcs are determined combinatorially
from the planar diagram.
Conversely, any configuration satisfying the conditions below
determines a four-page presentation of the link.

\begin{definition}~\label{def:binding}
Let $D$ be a knot or link diagram.
A simple closed curve $\gamma$ in the plane is called a
\emph{binding circle for a circular four-page presentation} of $D$
if the following conditions hold:
\begin{enumerate}
    \item The curve $\gamma$ intersects $D$ in finitely many points.
    \item Each arc of $D$ cut by $\gamma$ is of one of the following four types:
        \begin{enumerate}
            \item it lies inside $\gamma$ and every crossing on the arc is an over-crossing;
            \item it lies inside $\gamma$ and every crossing on the arc is an under-crossing;
            \item it lies outside $\gamma$ and every crossing on the arc is an over-crossing;
            \item it lies outside $\gamma$ and every crossing on the arc is an under-crossing.
        \end{enumerate}
    \item No two arcs of the same type are adjacent along $\gamma$.
\end{enumerate}
\end{definition}

An intersection point of $\gamma$ with $D$ is called a
\emph{binding point}.
The number of binding points is equal to the number of arcs
in the associated four-page presentation.

Since the number of binding points equals the number of arcs,
it suffices to construct, for a diagram $D$,
a binding circle $\gamma$ and count its intersections with $D$.
Accordingly, we regard $D$ as a 4-valent plane graph
and construct a binding circle that meets each edge
of the underlying graph at most once.
This yields a combinatorial bound on the number of arcs.

\section{Four-page indices from knot diagrams}
\label{sec:diagram}

In this section we prove Theorem~\ref{thm:alpha4_bound}.
Suppose first that $K$ is a split link whose non-split components
each satisfy Theorem~\ref{thm:alpha4_bound}.
Since the crossing number is additive under split union,
and four-page presentations of split components may be arranged
independently in disjoint regions of the open book,
the desired inequality follows by applying the argument
to each non-split component separately.

However, a trivial split component has crossing number zero,
while any four-page presentation of it requires at least two arcs.
Thus the inequality need not hold in the presence of trivial split components.
Accordingly, it suffices to treat the case where $K$ is non-split and nontrivial.

A crossing of a knot diagram $D$ is called \emph{nugatory} 
if there exists a simple closed curve in the plane 
that meets $D$ transversely in exactly one point 
and separates the crossing from the rest of the diagram.
A diagram is said to be \emph{reduced} 
if it contains no nugatory crossings.

Since nugatory crossings do not affect the knot type and
a minimal diagram contains no nugatory crossings,
we may assume that $D$ is reduced.
This assumption simplifies the combinatorial arguments 
in the subsequent construction.

Let $D$ be a non-split reduced diagram of a nontrivial knot or link $K$ with $c(D)$ crossings.
Ignoring the crossing information, $D$ may be regarded as a connected 4-valent plane graph.
Write $G(D)$ for the underlying 4-valent plane graph of $D$.
Since $G(D)$ is 4-valent, we have
$$
|E(G(D))| = 2|V(G(D))| = 2c(D).
$$
Our goal is to construct a binding circle $\gamma$
for a circular four-page presentation of $D$
such that $\gamma$ meets each edge of $G(D)$ exactly once.
This immediately implies that the number of binding points is $|E(G(D))|=2c(D)$,
and hence $\alpha_4(K)\le 2c(K)$.

Fix a checkerboard coloring of the regions of $D$.
Let $\Gamma(D)$ be the graph whose vertices are the shaded regions
and whose edges correspond to crossings of $D$.
Since $D$ is non-split, $\Gamma(D)$ is connected.
Choose a spanning tree $T$ of $\Gamma(D)$.

We now define a Kauffman state $S_T$ determined by the spanning tree $T$.
At each crossing of $D$, there are two possible smoothings~\cite{K2}.
We choose the smoothing so that the resulting state
encodes whether the corresponding dual edge lies in $T$ or in its complement.
More precisely, each crossing of $D$ corresponds to a unique edge of $\Gamma(D)$,
and we select the smoothing at that crossing according to membership in $T$.
Denote by $S_T$ the state obtained in this way.

\begin{lemma}\label{lem:singlecircle}
The state $S_T$ has exactly one state circle.
\end{lemma}

\begin{proof}
The chosen smoothings merge state circles along the edges of $T$
and prevent additional components from forming along edges outside $T$.
Equivalently, the resulting state circle may be identified with the boundary
of a regular neighborhood of a suitable spanning subcomplex of the induced
cell decomposition of $S^2$ determined by $T$.
Hence $S_T$ has exactly one component.
\end{proof}

Let $C_T$ denote the unique state circle of $S_T$.
Since $C_T$ runs along each edge of the underlying graph $G(D)$ exactly once,
it determines an Eulerian tour
$$
E_T = (e_1, \dots, e_{2c(D)})
$$
of $G(D)$.

For each edge $e$ of $G(D)$, choose a point in the interior of $e$,
for instance its midpoint, and denote it by $m(e)$.
Following the Eulerian tour $E_T$, we connect the successive midpoints
$m(e_i)$ and $m(e_{i+1})$
by embedded arcs drawn in the interior of the face of $G(D)$
where the two consecutive edges $e_i$ and $e_{i+1}$ meet.
In this way, the resulting curve remains disjoint from $G(D)$
except at the set of points $\{m(e) \,| \, e \in E(G(D))\}$.
Since $E_T$ is a cyclic sequence,
these arcs connect consecutively and return to $m(e_1)$,
so the construction produces a closed curve.
We denote this closed curve by $\gamma'$.

\begin{lemma}\label{lem:simple}
The curve $\gamma'$ may be chosen to be simple.
\end{lemma}

\begin{proof}
Each connecting arc lies in the interior of a face of $G(D)$.
Although several arcs may lie in the same face,
they can be arranged to be pairwise disjoint
by choosing them sufficiently close to the boundary of the face
and by performing small perturbations if necessary.
These adjustments do not alter the intersection pattern with $G(D)$.
\end{proof}

By construction of $\gamma'$, each crossing of $D$
is cut into two arcs, an over–arc and an under–arc,
and these arcs lie entirely on one side of $\gamma'$.
Moreover, $\gamma'$ meets $D$ in exactly $2c(D)$ points,
namely at the chosen points $m(e)$, one on each edge $e$ of $G(D)$.
Thus $\gamma'$ satisfies Condition~(1) of Definition~\ref{def:binding}
of a binding circle for a circular four-page presentation.

Observe that $\gamma'$ is obtained from the state circle $C_T$
by a planar isotopy which replaces each subarc of $C_T$
between successive points $m(e_i)$ and $m(e_{i+1})$
with the corresponding connecting arc.
In particular, $\gamma'$ is isotopic to $C_T$
in the complement of $G(D)$,
and hence separates the plane into the same two regions.
The next lemma shows that $\gamma'$ also satisfies
Condition~(2) of Definition~\ref{def:binding}.

\begin{lemma}\label{lem:inside_outside_crossings}
Both the interior and the exterior of $\gamma'$ contain
at least one crossing of $D$.
\end{lemma}

\begin{proof}
Since $D$ is non-split, the dual graph $\Gamma(D)$ is connected.
Because $D$ is reduced, $\Gamma(D)$ has no cut edges.
A connected graph without cut edges cannot be a tree.
Hence $\Gamma(D)$ is not a tree, and for any spanning tree
$T \subset \Gamma(D)$ there exists at least one edge of $\Gamma(D)$
not contained in $T$.

By construction of the state $S(T)$,
crossings corresponding to edges of $T$
lie on one side of the state circle $C_T$,
while crossings corresponding to edges
in $\Gamma(D)\setminus T$
lie on the other side.
Since $\gamma'$ is isotopic to $C_T$,
the same separation property holds for $\gamma'$.

Because $T$ contains at least one edge,
the interior of $\gamma'$ contains at least one crossing.
Because $\Gamma(D)\setminus T$ is nonempty,
the exterior also contains at least one crossing.
\end{proof}

\begin{figure}[h!]
	\includegraphics{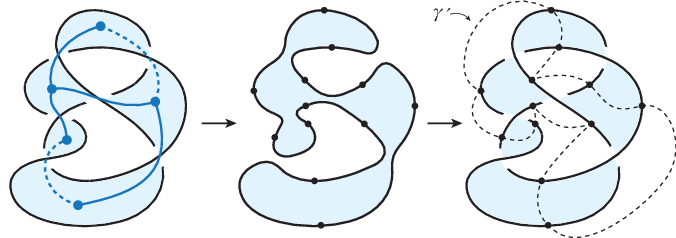}
	\caption{A binding circle for circular four-page presentation}
	\label{fig:middle_step}
\end{figure}

The construction described above is summarized in Figure~\ref{fig:middle_step}.
If the diagram $D$ is alternating, then along the simple closed curve $\gamma'$
the over- and under-passes necessarily occur in an alternating fashion.
Consequently, $\gamma'$ satisfies Condition~(3) of Definition~\ref{def:binding}.
In this case, the curve $\gamma'$ satisfies all the conditions of
Definition~\ref{def:binding}, and hence defines a binding circle.
For a non-alternating diagram, however, this adjacency condition may fail,
and an additional local modification is required.
We therefore prove the following lemma.

\begin{lemma}\label{lem:local_fix_nonalt}
Let $D$ be a link diagram and let $\gamma'$ be the closed curve constructed from
the Eulerian tour as above.
Suppose that an edge $e_i$ of $G(D)$ has the property that its two endpoint crossings
are non-alternating, and that both of these crossings lie on the same side of $\gamma'$
(either both in the interior or both in the exterior).
Then $\gamma'$ may be modified by a planar isotopy supported in a small neighborhood of
$e_i$ so that the modified curve $\widetilde{\gamma}$ is disjoint from $m(e_i)$
and still satisfies Condition~(1) and~(2) of Definition~\ref{def:binding}.
\end{lemma}

\begin{proof}
Let $N$ be a sufficiently small closed regular neighborhood of the edge $e_i$ in the plane.
Choose $N$ so that it meets $D$ exactly in the portion of $D$ consisting of $e_i$
together with small subarcs near its two endpoint crossings, and so that $N$ contains
no other crossings or edges of $G(D)$.
By construction, $\gamma'$ meets $e_i$ only at the chosen point $m(e_i)$.
After possibly shrinking $N$, we may assume that $\gamma' \cap N$ is a single embedded arc
meeting $e_i$ at $m(e_i)$ and connecting two points on $\partial N$.

Inside $N$, replace the subarc $\gamma' \cap N$ by another embedded arc in $N$
with the same endpoints on $\partial N$ but running along $\partial N$,
so that it avoids the point $m(e_i)$ and stays in $N\setminus D$ except possibly
for transverse intersections with the same edges of $G(D)$ as before.
Denote the resulting closed curve by $\widetilde{\gamma}$.
Outside $N$ we leave $\gamma'$ unchanged.

By the choice of $N$, the modification does not create any new intersections of
$\widetilde{\gamma}$ with $D$ outside $N$, and within $N$ the replacement arc is chosen
to remain in the complement of $D$.
Hence $\widetilde{\gamma}$ still intersects $D$ in finitely many points, so it satisfies
Condition~(1) of Definition~\ref{def:binding}.
Moreover, because the modification is supported in a neighborhood containing only the
local portion of $D$ near $e_i$, each arc of $D$ cut by $\widetilde{\gamma}$ is of the
same type as before (inside/outside and over/under behavior along that arc is unchanged).
In particular, Condition~(2) of Definition~\ref{def:binding} is preserved.

Therefore $\widetilde{\gamma}$ satisfies Conditions~(1) and~(2) of
Definition~\ref{def:binding}, as claimed.
\end{proof}

Applying Lemma~\ref{lem:local_fix_nonalt} repeatedly, we conclude that even in the non-alternating case the curve can be modified so as to satisfy Condition~(3) of Definition~\ref{def:binding} as drawn in Figure~\ref{fig:final_step}.

\begin{figure}[h!]
	\includegraphics{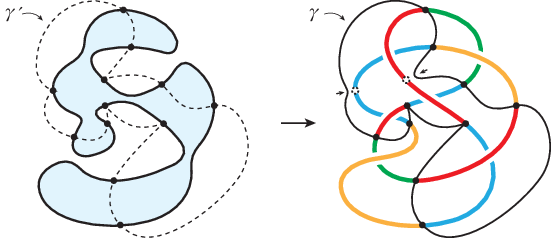}
	\caption{A binding circle for circular four-page presentation}
	\label{fig:final_step}
\end{figure}

We are now ready to complete the proof of Theorem~\ref{thm:alpha4_bound}.

\begin{proof}[Proof of Theorem~\ref{thm:alpha4_bound}]
Let $D$ be a non-split reduced diagram of a nontrivial knot or link $K$
with $c(D)$ crossings.
By the construction above, we obtain a simple closed curve $\gamma$
which serves as a binding circle for a circular four-page presentation of $D$.
Moreover, $\gamma$ meets each edge of the underlying graph $G(D)$
at most once.
Since $G(D)$ has $2c(D)$ edges, the number of binding points is at most $2c(D)$.
Hence the associated four-page presentation has at most $2c(D)$ arcs, and therefore
$$
\alpha_4(K) \le 2c(D).
$$
Taking the minimum over all diagrams of $K$ yields
$$
\alpha_4(K) \le 2c(K).
$$

If $D$ is non-alternating, then there exists at least one edge of $G(D)$ whose two endpoint crossings are non-alternating.
Choose a spanning tree $T \subset \Gamma(D)$ so that the edge of $\Gamma(D)$ corresponding to this crossing configuration is not contained in $T$ (equivalently, the two endpoint crossings lie on the same side of the state circle determined by $T$).
With this choice of spanning tree, the resulting curve $\gamma'$ satisfies the hypotheses of Lemma~\ref{lem:local_fix_nonalt}.
Applying that lemma, we may locally modify $\gamma'$ near the midpoint of the corresponding edge so as to remove one binding point while preserving Conditions~(1) and~(2) of Definition~\ref{def:binding}.
Consequently, in the non-alternating case the number of binding points can be reduced by at least one, while maintaining the defining conditions of a binding circle.
In this case we obtain
$$
\alpha_4(K) < 2c(K).
$$
This completes the proof.
\end{proof}

Theorem~\ref{thm:alpha4_bound} provides a linear upper bound with coefficient $2$.
We now show that this coefficient is sharp for certain links.

\begin{corollary}
$\alpha_4(2^2_1)=4$ and $\alpha_4(3_1)=6$.
\end{corollary}

\begin{proof}
For the Hopf link $2^2_1$, we have $\alpha(2^2_1)=4$ and $c(2^2_1)=2$.
Since
$$
\alpha(K) \le \alpha_4(K) \le 2c(K),
$$
it follows that
$$
4 \le \alpha_4(2^2_1) \le 4,
$$
and hence $\alpha_4(2^2_1)=4$.

For the trefoil knot $3_1$, we have $\alpha(3_1)=5$ and $c(3_1)=3$.
Thus
$$
5 \le \alpha_4(3_1) \le 6.
$$

Suppose, for contradiction, that $\alpha_4(3_1)=5$.
Then there exists a circular four-page presentation of $3_1$
with exactly five binding points.

By Lemma~\ref{lem:inside_outside_crossings},
both the interior and the exterior of the binding circle
must contain at least one crossing of the diagram.
Since the trefoil knot has three crossings,
at least one of the regions determined by the binding circle
must contain at least two crossings.

However, by Condition~(3) of Definition~\ref{def:binding},
no two arcs of the same type may be adjacent along the binding circle.
To separate these crossings into distinct arc types
while preserving the alternating pattern along the circle,
at least six binding points are required.
Hence a circular four-page presentation with only five binding points
cannot exist.
Therefore $\alpha_4(3_1)=6$.
\end{proof}

\section*{Acknowledgements}
This study was supported by Basic Science Research Program of the National Research Foundation of Korea (NRF) grant funded by the Korea government Ministry of Education (RS-2023-00244488).

\bibliography{ribgen.bib} 
\bibliographystyle{abbrv}

\end{document}